\documentclass[11pt,a4paper]{amsart}
\usepackage{fancyhdr}
\usepackage{a4wide}
\usepackage{graphicx}
\usepackage{float}
\usepackage{amssymb}
\usepackage{amsmath}
\usepackage{amsthm}
\usepackage{color}
\usepackage{mathrsfs}
\usepackage[T1]{fontenc}
\usepackage{inputenc}
\usepackage[english]{babel}
\usepackage{lmodern}
\usepackage{hyperref}
\usepackage{geometry}
\usepackage{changepage}
\usepackage{xcolor,colortbl}
\changepage{0pt}{}{}{}{}{0pt}{}{0pt}{10pt}
\usepackage[numbers]{natbib}
\usepackage{hyperref}
\hypersetup{
pdfpagemode=UseThumbs,
pdftoolbar=true,        
pdfmenubar=true,        
pdffitwindow=false,     
pdfstartview={Fit},    
pdftitle={Construction of invisible conductivities for the point electrode model in electrical impedance tomography},    
pdfauthor={L. Chesnel, N. Hyv\"{o}nen, S. Staboulis},     
pdfcreator={L. Chesnel, N. Hyv\"{o}nen, S. Staboulis},   
pdfproducer={L. Chesnel, N. Hyv\"{o}nen, S. Staboulis}, 
pdfkeywords={}, 
pdfnewwindow=true,      
colorlinks=true,       
linkcolor=magenta,          
citecolor=red,        
filecolor=cyan,      
urlcolor=blue           
}

\newcommand{\dsp}{\displaystyle}
\newcommand{\eps}{\varepsilon}
\newcommand{\om}{\omega}
\newcommand{\Om}{\Omega}
\newcommand{\mrm}[1]{\mathrm{#1}}

\newcommand{\R}{\mathbb{R}}
\renewcommand{\div}{\mrm{div}}
\newcommand{\mL}{\mrm{L}}
\newcommand{\mH}{\mrm{H}}

\renewcommand{\ker}{\mrm{ker}}

\definecolor{Gray}{gray}{0.90}

\newcommand{\overbar}[1]{\mkern 1.0mu\overline{\mkern-1.0mu#1\mkern-1.0mu}\mkern 1.0mu}

\newtheorem{theorem}{Theorem}[section]
\newtheorem{lemma}{Lemma}[section]
\newtheorem{remark}{Remark}[section]

\newtheorem{proposition}{Proposition}[section]

\newtheorem{alg}{Algorithm}

\begin{document}
\title[Invisible conductivity perturbations in EIT]{Construction of invisible conductivity perturbations for the point  electrode model in electrical impedance tomography}

\author{Lucas Chesnel}
\address{Centre de Math\'ematiques Appliqu\'ees, bureau 2029, \'Ecole Polytechnique, 91128 Palaiseau Cedex, France}
\email{lucas.chesnel@cmap.polytechnique.fr}

\author{Nuutti Hyv\"onen}
\address{Aalto University, Department of Mathematics and Systems Analysis, P.O. Box 11100, FI-00076 Aalto, Finland}
\email{nuutti.hyvonen@aalto.fi}

\author{Stratos Staboulis}
\address{University of Helsinki, Department of Mathematics and Statistics, P.O. Box 68, FI-00014, Finland}
\email{stratos.staboulis@helsinki.fi}



\begin{abstract}
We explain how to build invisible isotropic conductivity perturbations of the unit conductivity in the framework of the point electrode model for two-dimensional electrical impedance tomography. The theoretical approach, based on solving a fixed point problem, is constructive and allows the implementation of an algorithm for approximating the invisible perturbations. The functionality of the method is demonstrated via numerical examples.
\end{abstract}

\maketitle

\section{Introduction}
{\em Electrical impedance tomography} (EIT) is a noninvasive imaging technique with applications, e.g., in medical imaging, process tomography, and nondestructive testing of materials \cite{Borcea02,Cheney99,Uhlmann09}. The aim  of EIT is to reconstruct the conductivity  distribution inside the examined physical body $D \subset \R^d$, $d \geq 2$, from boundary measurements of current and potential. From the purely theoretical standpoint, the inverse problem of EIT, also known as the inverse conductivity problem, corresponds to determining the strictly positive conductivity $\sigma: D \to \R$ in the elliptic equation
\begin{equation}
\label{basic}
\div(\sigma \nabla u)  =  0 \quad\mbox{in }D 
\end{equation}
from knowledge of the corresponding Neumann-to-Dirichlet (or Dirichlet-to-Neumann) map at the object boundary $\partial D$. This formulation corresponds to the idealized {\em continuum model} (CM), which is mathematically attractive in its simplicity but lacks a straightforward connection to practical EIT measurements that are always performed with a finite number of contact electrodes. On the other hand, it is widely acknowledged that the most accurate model for real-life EIT is the {\em complete electrode model} (CEM), which takes into account electrode shapes and contact resistances at electrode-object interfaces \cite{Cheng89,Somersalo92}.

Instead of the CM or the CEM, in this work we employ the {\em point electrode model} (PEM) that treats the current-feeding electrodes of EIT as delta-like boundary current sources and models the potential measurements as pointwise evaluations of the corresponding solution to \eqref{basic} at the electrode locations~\cite{Hanke11}. (To make such a model well defined, relative potential measurements need to be considered due to the singularities induced by the localized current inputs; see Section~\ref{sec:setting} for the details.) The PEM is a reasonably good model for practical EIT measurements --- in particular, superior to the CM (cf.~\cite{Hyvonen09}) --- if the electrodes are small compared to the size of the imaged body: The discrepancy between the PEM and the CEM is of the order $O(h^2)$, with $h>0$ being the maximal diameter of the electrodes \cite{Hanke11}. In this work we develop a method, introduced by \cite{BoNa13} in the context of acoustic waveguides (see also \cite{Naza11} for the original idea), for building invisible perturbations of a given reference conductivity for an arbitrary but fixed electrode configuration modeled by the PEM.

To put our work into perspective, let us briefly review the main results on the unique solvability of the inverse conductivity problem, which was proposed to the mathematical community by Calder\'on in \cite{Calderon80}. For $d\geq3$, the global uniqueness for $\mathscr{C}^2$-conductivities was proven in \cite{Sylvester87}; an extension of this result to the case Lipschitz conductivities can be found in \cite{Tataru13}. The first global uniqueness proof in two spatial dimensions was given by \cite{Nachman96} for $\mathscr{C}^2$-conductivities; subsequently, \cite{Astala06} proved uniqueness for general  $L^\infty$-conductivities.  All these articles assume the Cauchy data for \eqref{basic} are known on all of $\partial D$, but the partial data problem of having access only to some subset(s) of $\partial D$ has also been tackled by many mathematicians; see, e.g., \cite{Imanuvilov10,Kenig07} and the references therein. From our view point, the most important uniqueness results are presented in \cite{Hyvonen13,Seis14} where it is shown that any two-dimensional conductivity that equals a known constant in some interior neighborhood of $\partial D$ is uniquely determined by the PEM measurements with countably infinite number of electrodes. This manuscript complements \cite{Hyvonen13,Seis14} by constructively showing that for an arbitrary, but fixed and finite, electrode configuration there exists a perturbation of the unit conductivity that is invisible for the EIT measurements in the framework of the PEM. For completeness, it should also be mentioned that recently the nonuniqueness of the inverse conductivity problem has been studied for {\em finite element method} (FEM) discretizations with piecewise linear basis functions \cite{Lassas13}. (Take note that we restrict our attention exclusively to isotropic conductivities because it is well known that the inverse conductivity problem is not uniquely solvable for anisotropic conductivities; see, e.g., \cite{Astala05,Sylvester90} and the references therein.)

Our constructive proof for the existence of `invisible conductivities' is based on introducing a small perturbation of a given reference conductivity and applying a suitable fixed point iteration. The general idea originates from \cite{CaNR12,Naza11,Naza11c,Naza11b,Naza12,Naza13}, where the authors develop a method for constructing small regular and singular perturbations of a waveguide that preserve the multiplicity of the point spectrum on a given interval of the continuous spectrum. Subsequently, the same approach has been adapted in \cite{BoNa13} (see also \cite{BoNTSu} for an application to a water wave problem) to demonstrate the existence of regular perturbations of a waveguide allowing waves at given frequencies to pass through without any distortion or with only a phase shift. Recently, in \cite{BoCNSu} the methodology has also been used in inverse obstacle scattering to construct defects in a reference medium which are invisible to a finite number of far field measurements. Although our technique is, in principle, applicable in any spatial dimension $d\geq2$ and for an arbitrary smooth enough reference conductivity, a necessary intermediate result about linear independence of certain auxiliary functions (cf.~Lemma~\ref{lemmaLinearIndep}) is proved here only for $d=2$ and the unit reference conductivity, making the main theoretical result of this work two-dimensional. We implement the constructive proof as a numerical algorithm that is capable of producing conductivities that are indistinguishable from the unit conductivity for a given set of point electrodes. We also numerically demonstrate that such conductivities are almost invisible for small electrodes within the CEM as well.

It should be emphasized that the existence of invisible conductivity perturbations for a finite number of electrodes is not very surprising: The space of admissible conductivities is infinite dimensional whereas EIT measurements with $N+1$ electrodes include only $N(N+1)/2$ degrees of freedom. Be that as it may, we are not aware of any previous works providing examples on this nonuniqueness that is of importance in practical EIT.

This text is organized as follows. Section~\ref{sec:setting} introduces the setting for our analysis and a representation of the relative PEM measurements with respect to a certain simple basis of electrode current patterns. The general scheme for constructing invisible conductivity perturbations is presented in Section~\ref{sec:scheme} and, subsequently, Section~\ref{sec:2D} adds the missing theoretical piece in two dimensions. The main result of the article is presented in Theorem \ref{MainTheorem}.
The algorithmic implementation of the constructive existence proof is considered in Section~\ref{sec:algo} and, finally, the numerical examples are presented in Section~\ref{sec:numerics}.

\section{Setting}
\label{sec:setting}
\noindent Let $D\subset\R^d$, $d\ge2$, be a simply connected and bounded domain with a $\mathscr{C}^{\infty}$-boundary. Introduce  a real valued \textit{reference} conductivity $\sigma^0\in\mathscr{C}^{\infty}(\overline{D})$ such that $\sigma^0\ge c > 0$ in $D$. 
(Throughout this text $C,c>0$ denote generic constants that may change from one occurrence to the next.) 
Consider the Neumann boundary value problem
\begin{equation}\label{BVP0}
\div(\sigma^0\nabla u^0)  =  0 \quad\mbox{in }D,\qquad \nu \cdot \sigma^0\nabla u^0=f\quad\mbox{on }\partial D
\end{equation}
for a current density $f$ in 
\begin{equation}\label{def0MeanSpace}
\mH^s_{\diamond}(\partial D) := \{g\in\mH^s(\partial D)\,|\,\langle g,1\rangle_{\partial D} =0\},
\end{equation}
with some $s\in\R$. Here and in what follows, $\nu$ denotes the unit normal vector of $\partial D$ orientated to the exterior of $D$. We observe that the dual of $\mH^s_{\diamond}(\partial D)$ is realized by 
\begin{equation}\label{defQSpace}
\mH^{-s}(\partial D)/\R := \mH^{-s}(\partial D)/\mrm{span}\{1\},\qquad s\in\R.
\end{equation}
It follows from standard theory of elliptic boundary value problems that (\ref{BVP0}) has a unique solution $u^0 \in \mH^{s+3/2}(D)/\R$, which is smooth away from $\partial D$ due to the interior regularity for elliptic equations. In particular, for any compactly embedded domain $\Omega \Subset D$ it holds that
\begin{equation}
\label{interiorR}
c \|u^0\|_{\mH^1(\Omega) / \R} \leq \|u^0\|_{\mH^{s+3/2}(D) / \R} \leq  C \| f \|_{\mH^s_{\diamond}(\partial D)},
\end{equation}
where $c = c(s,\Omega)$ and $C = C(s)$ \cite{LiMa68}. Let us next consider a \textit{perturbed} conductivity 
\begin{equation}\label{perturbedConductivity}
\sigma^{\eps}:=\sigma^{0}+\eps \kappa,
\end{equation}
where $\kappa\in\mrm{L}^{\infty}(D)$ is compactly supported in $D$ and $\eps>0$ is such that $\sigma^{\eps}\ge c>0$ in $D$. The corresponding perturbed Neumann boundary value problem
\begin{equation}\label{BVPperturbed}
\div(\sigma^\eps\nabla u^{\eps})  =  0 \quad\mbox{in }D,\qquad \nu \cdot \sigma^\eps\nabla u^{\eps}=f\quad\mbox{on }\partial D,
\end{equation}
with $f \in \mH^s_{\diamond}(\partial D)$ for some $s \in \R$, has a unique solution in $(\mH^{s+3/2}(D)\cap\mH^{1}_{\mrm{loc}}(D))/\R$\footnote{Here, $\mH^{1}_{\mrm{loc}}(D)$ stands for the space of distributions $v \in \mathscr{D}'(D)$ verifying $v|_{\om}\in\mH^{1}(\om)$ for all domains $\om \Subset D$.} (cf., e.g., \cite{LiMa68,Hanke11b}).

We note that the relative Neumann-to-Dirichlet map 
\[
\Lambda^\eps-\Lambda^0 : f\mapsto (u^\eps-u^0)|_{\partial D},\qquad \mathscr{D}'_{\diamond}(\partial D)\to \mathscr{D}(\partial D)/\R
\] 
is well-defined. Here, the mean-free distributions $\mathscr{D}'_{\diamond}(\partial D)$ and the quotient space of smooth boundary potentials $\mathscr{D}(\partial D)/\R$ are defined in accordance with (\ref{def0MeanSpace}) and (\ref{defQSpace}). This regularity result can be deduced from standard elliptic theory (cf., e.g., \cite{LiMa68,Hanke11b}) and follows from the fact that $\Lambda^\eps$ and $\Lambda^0$ are pseudodifferential operators with the same symbol on $\partial D$ because $\sigma^{\eps}-\sigma^{0}$ vanishes in some interior neighborhood of $\partial D$. One can also prove the following symmetry property for the relative Neumann-to-Dirichlet map (see, e.g., \cite[Theorem 2.1]{Seis14}).
\begin{proposition}\label{PropoSymRelDTN}
For all $\varphi$, $\varphi'\in\mathscr{D}'_{\diamond}(\partial D)$, we have
\[
\langle \varphi,(\Lambda^\eps-\Lambda^0)\varphi'\rangle_{\partial D}  = \langle \varphi',(\Lambda^\eps-\Lambda^0)\varphi\rangle_{\partial D}  = \dsp\int_{D}(\sigma^0-\sigma^\eps)\nabla u^\eps_{\varphi}\cdot \nabla u^0_{\varphi'} \,d x,
\]
where $u^\eps_{\varphi}$, $u^0_{\varphi'}$  denote respectively the solutions of (\ref{BVPperturbed}), (\ref{BVP0}) with $f$ equal to $\varphi$, $\varphi'$.
\end{proposition}
Now, consider an observer who can impose currents between pairs of (small) electrodes located at $x_0,\dots, x_{N} \in \partial D$ and measure the resulting voltages for both conductivities $\sigma^0$ and $\sigma^\eps$. According to the PEM \cite{Hanke11}, this corresponds to knowing all elements in the matrix of relative measurements $\mathscr{M}(\sigma^\eps)\in\R^{N\times N}$ defined via
\begin{equation}\label{defMeasure}
\mathscr{M}_{i j}(\sigma^\eps) = \langle \delta_i-\delta_{0},(\Lambda^\eps-\Lambda^0)(\delta_j-\delta_{0})\rangle_{\partial D}, \qquad i,j = 1, \dots, N,
\end{equation}
where $\delta_{n} \in \mH^{-(d-1)/2-\eta}(\partial D)$, $\eta > 0$, stands for the Dirac delta distribution supported at $x_n$. It follows from  Proposition~\ref{PropoSymRelDTN} that $\mathscr{M}(\sigma^\eps)$ is symmetric and, furthermore,
\begin{equation}\label{eqnMeasure}
\mathscr{M}_{i j}(\sigma^\eps) =  \dsp\int_{D}(\sigma^0-\sigma^\eps)\nabla u^\eps_i\cdot \nabla u^0_j \,d x,
\end{equation}
where 
$u^0_n, u^{\eps}_n\in (\mH^{-(d-4)/2-\eta}(D)\cap\mH^{1}_{\mrm{loc}}(D))/\R$, $\eta>0$, are the solutions of the Neumann problems
\begin{equation}\label{pbRef0}
\div(\sigma^0\nabla u^0)  =  0 \quad\mbox{in }D,\qquad \nu \cdot \sigma^0\nabla u^0= \delta_{n} - \delta_{0}\quad\mbox{on }\partial D
\end{equation}
and
\begin{equation}\label{pbRef}
\div(\sigma^\eps\nabla u^{\eps})  =  0 \quad\mbox{in }D,\qquad \nu \cdot \sigma^\eps\nabla u^{\eps}=\delta_{n} - \delta_{0}\quad\mbox{on }\partial D,
\end{equation}
respectively. The goal of this work is to find $\sigma^{\eps}\not\equiv\sigma^{0}$ such that $\mathscr{M}(\sigma^\eps)$ vanishes.

\begin{remark}
\label{remark}
In the framework of the PEM, the relative EIT measurements corresponding to the electrode locations $x_0,\dots, x_{N} \in \partial D$ are, arguably, more intuitively described by the measurement map
\begin{equation}
\label{Mope}
M(\sigma^\eps): I \mapsto \big[(u_I^\eps - u_I^0)(x_n)\big]_{n=0}^{N}, \quad \R^{N+1}_{\diamond} \to \R^{N+1}/\R \simeq \R^{N+1}_{\diamond},
\end{equation}
where $\R^{N+1}_{\diamond}$ denotes the mean-free subspace of $\R^{N+1}$, and $u_I^\eps$ and $u_I^\eps$ are the solutions of \eqref{BVP0} and \eqref{BVPperturbed}, respectively, for the boundary current density
$$
f = \sum_{n=0}^{N} I_n \delta_n \in \mH^{-(d-1)/2-\eta}_\diamond(\partial D), \qquad \eta > 0.
$$
In other words, $M(\sigma^\eps) I \in \R^{N+1}/\R$ contains the relative potentials at the electrodes (uniquely defined up to the ground level of potential) when the net currents $I \in \R^{N+1}_\diamond$ are driven through the boundary points $x_0,\dots, x_{N} \in \partial D$.
It is easy to see that the matrix $\mathscr{M}(\sigma^\eps) \in \R^{N \times N}$ is the representation of $M(\sigma^\eps): \R^{N+1}_{\diamond} \to  \R^{N+1}_{\diamond}$ with respect to the basis 
$$
{\rm e}^n - {\rm e}^0, \quad n=1,\dots,N,
$$
where ${\rm e}^n$ denotes the $n$th Cartesian basis vector of $\R^{N+1}$. In particular, $\mathscr{M}(\sigma^\eps)$ is the null matrix if and only if $M(\sigma^\eps)$ is the null operator, that is, if and only if $\sigma^\eps$ cannot be distinguished from $\sigma^0$ by PEM measurements.
\end{remark}

\section{General scheme}
\label{sec:scheme}
Let $\Om \Subset D \subset \R^d$ be a nonempty Lipschitz domain. We  recall the decomposition $\sigma^{\eps} = \sigma^{0}+\eps \kappa$ and require that $\kappa \in L^\infty(D)$ satisfies  $\mrm{supp}(\kappa)\subset \overline{\Om}$. Our leading idea, originating from \cite{BoCNSu,BoNa13,Naza11}, is to consider a small enough $\eps > 0$ so that we can compute an asymptotic expansion for $\mathscr{M}(\sigma^\eps)$. To this end, let us first write
\begin{equation}\label{decompoUeps}
u^\eps_n=u^0_n+\eps\tilde{u}^{\eps}_n
\end{equation}
where $u^{\eps}_n, \, u^0_n \in (\mH^{-(d-4)/2-\eta}(D)\cap\mH^{1}_{\mrm{loc}}(D))/\R$ are defined by (\ref{pbRef}) and \eqref{pbRef0}, respectively.
Plugging (\ref{decompoUeps}) in \eqref{pbRef}, we see that $\tilde{u}^{\eps}_n$ must satisfy the Neumann problem
\begin{equation}\label{pbDefOrder1}
\div(\sigma^\eps \nabla \tilde{u}^{\eps}_n)  =  -  \div(\kappa\nabla u^0_n ) \quad\mbox{in }D,\qquad \nu \cdot \sigma^0\nabla \tilde{u}^\eps = 0\quad\mbox{on }\partial D.
\end{equation}
Since $\div(\kappa\nabla u^0_n)$ defines a compactly supported source in $\mH^{-1}(D)$, by applying the Lax--Milgram lemma to the variational formulation of \eqref{pbDefOrder1}, we see that \eqref{pbDefOrder1} has a unique solution, i.e.~$\tilde{u}^{\eps}_n$, in $\mH^1(D) / \R$.\footnote{Note that $\| \nabla \cdot \|_{\mL^2(D)}$ is equivalent to the standard quotient norm of $\mH^1(D) / \R$ by the Poincar\'e inequality.} Moreover, the Lax--Milgram lemma and \eqref{interiorR} also imply that
\begin{equation}
\label{pertnorm}
\| \tilde{u}^{\eps}_n \|_{\mH^1(D) / \R} \leq C \| (\sigma^\eps)^{-1} \|_{\mL^\infty(D)} \| \kappa \|_{\mL^\infty(\Omega)} \| u^0_n \|_{\mH^1(\Omega) / \R} 
\leq C \| (\sigma^\eps)^{-1} \|_{\mL^\infty(D)} \| \kappa \|_{\mL^\infty(\Omega)}
\end{equation}
where the latter $C>0$ depends on $\Omega$. 

Inserting (\ref{decompoUeps}) in (\ref{eqnMeasure}), we deduce that  
\begin{equation}\label{termInter}
\mathscr{M}_{i j}(\sigma^\eps) = -\eps \int_{D} \kappa\,\nabla u^0_i\cdot \nabla u^0_j \,d x-\eps^2 \int_{D} \kappa\,\nabla \tilde{u}^{\eps}_i\cdot \nabla u^0_j \,d x.
\end{equation}
The representation \eqref{termInter} demonstrates that $\mathscr{M}(\sigma^\eps)$ is of the order $\eps$, which was to be expected as $\mathscr{M}(\sigma^0)$ is the null matrix and $\sigma^\eps-\sigma^0 = \eps \kappa$. More interestingly,  the first term in the asymptotic expansion of $\mathscr{M}(\sigma^\eps)$ has a simple linear dependence on $\kappa$, which makes it relatively easy to find a nontrivial $\kappa$ such that $\mathscr{M}(\sigma^\eps)$ is of the order $\eps^2$. Unfortunately, the higher order terms in $\eps$ depend less explicitly on $\kappa$. To cope with this difficulty, we will next introduce a suitable fixed point problem. 
 
We redecompose $\kappa$ in the form
\begin{equation}\label{ExpressKappa}
\kappa = \kappa_0 + \sum_{j=1}^{N}\sum_{i=1}^{j}\tau_{ij}\,\kappa_{ij}
\end{equation}
where $\tau_{ij} \in \R$ are free parameters. Moreover, $\kappa_0,\,\kappa_{ij} \in L^\infty(D)$ are fixed functions that are supported in $\overline{\Omega}$ and assumed to satisfy the conditions
\begin{equation}\label{Conditions1}
\int_{D} \kappa_0\,\nabla u^0_{i'}\cdot \nabla u^0_{j'}\,d x=0\qquad \mbox{for }1\le i'\le j' \le N
\end{equation}
and
\begin{equation}\label{Conditions2}
\int_{D} \kappa_{ij} \,\nabla u^0_{i'}\cdot \nabla u^0_{j'}\,d x= \begin{array}{|ll} 1 & \mbox{ if } (i,j)=(i',j')\mbox{ or }(j,i)=(i',j'),\\[1mm]
0 & \mbox{ else}.\\
\end{array}
\end{equation}
The construction of such $\kappa_{ij}$ will be considered in Lemma~\ref{lemmaLinearIndep} and Section~\ref{sec:2D} below, but meanwhile we just assume they exist.
We remark that it seems reasonable to assume that the $N(N+1)/2$ free parameters $\tau_{ij}$ in the perturbation $\kappa$ are enough to cancel out the symmetric matrix $\mathscr{M}(\sigma^\eps) \in \R^{N \times N}$ (cf.~Proposition \ref{PropoSymRelDTN}). Substituting (\ref{ExpressKappa}) in (\ref{termInter}) and using (\ref{Conditions1})--(\ref{Conditions2}), we obtain the expansion
\begin{equation}\label{FFPTotal}
\mathscr{M}(\sigma^\eps) = -\eps\,\tau -\eps^2\,\tilde{\mathscr{M}}^\eps(\tau)
\end{equation}
where
\begin{equation}\label{defTermContra}
\tilde{\mathscr{M}}^\eps_{ij}(\tau) = \tilde{\mathscr{M}}^\eps_{ji}(\tau) = \int_{D} \kappa\,\nabla \tilde{u}^{\eps}_i\cdot \nabla u^0_j \,d x, \qquad 1\le i\le j \le N,
\end{equation}
and $\tau \in \R^{N \times N}$ is the symmetric matrix defined by the parameters  $\tau_{ij}$. 

From (\ref{FFPTotal}) it is obvious that imposing $\mathscr{M}(\sigma^\eps)=0$ in the chosen setting is equivalent to solving the following fixed point problem: 
\begin{equation}\label{PbFixedPoint}
\begin{array}{|l}
\mbox{Find }\tau\in S^N\mbox{ such that }\tau = F^{\eps}(\tau),
\end{array}~\\[5pt]
\end{equation}
where  $S^N$ denotes the space of symmetric  $N\times N$ matrices and $F^{\eps}: S^N \to S^N$ is defined by  
\begin{equation}\label{DefMapContra}
F^{\eps}(\tau)= -\eps\,\tilde{\mathscr{M}}^\eps(\tau). 
\end{equation}
Lemma \ref{lemmaTech} below ensures that for any fixed parameter $\gamma>0$ and
a small enough $\eps > 0$, the map $F^\eps$ has the invariant set
\begin{equation}
\label{Bg}
\mathbb{B}_\gamma:=\{\tau\in S^{N}\,\big|\,|\tau| \le \gamma\}
\end{equation}
on which it is a contraction. In \eqref{Bg}, $|\cdot|$ denotes an arbitrary norm of $S^{N}$. In consequence, the Banach fixed point theorem guarantees the existence of $\eps_0 = \eps_0(\gamma) >0$ such that for all $\eps\in(0;\eps_0]$, the fixed point problem (\ref{PbFixedPoint}) has a unique solution in $\mathbb{B}_\gamma$, enabling the construction of $\sigma^\eps$ for which $\mathscr{M}(\sigma^\eps) = 0$ via \eqref{ExpressKappa} and \eqref{perturbedConductivity}. 

It is important to notice that the constructed $\kappa$ defined by (\ref{ExpressKappa}) verifies $\kappa\not\equiv0$ whenever $\kappa_0\not\equiv0$ by virtue of the orthogonality conditions (\ref{Conditions1})--(\ref{Conditions2}). In other words, if $\kappa_0 \not\equiv0$, then also $\sigma^\eps \not\equiv \sigma^0$ as required. Moreover, for given $\sigma^0$, $\kappa_0$, $\kappa_{ij}$ and $\gamma>0$, the upper bound $\eps_0 > 0$ can be tuned to ensure that the perturbed conductivity corresponding to $\kappa$ of \eqref{ExpressKappa} satisfies $\sigma^\eps \geq c > 0$ for all $\eps \in [0;\eps_0]$ and $\tau \in \mathbb{B}_\gamma$, which guarantees that all conductivities involved in the fixed point iteration are admissible.

In particular, we have proved the following result (modulo Lemma~\ref{lemmaTech}):
\begin{proposition}\label{propoConstruction}
Let $\Om\Subset D$ be a Lipschitz domain. Assume that there are functions $\kappa_0, \, \kappa_{ij} \in \mL^\infty(D)$ supported in $\overline{\Om}$ satisfying (\ref{Conditions1})-(\ref{Conditions2}). Then, there exists a conductivity $\sigma^{\eps}\in\mrm{L}^{\infty}(D)$, with $\sigma^\eps \geq c > 0$,  such that $\sigma^{\eps}-\sigma^{0}\not\equiv0$, $\mbox{supp}(\sigma^{\eps}-\sigma^{0})\subset \overline{\Om}$ and
\[
\mathscr{M}_{ij}(\sigma^\eps) = \langle \delta_i-\delta_{0},(\Lambda^\eps-\Lambda^0)(\delta_j-\delta_{0})\rangle_{\partial D}=0
\]
for all $i,j=1,\dots,N$.
\end{proposition}

\begin{remark}
If $\kappa_0$, $\kappa_{ij}$ supported in $\overline{\Om}$ and satisfying (\ref{Conditions1})-(\ref{Conditions2}) exist, Proposition~\ref{propoConstruction} indicates that the conductivities $\sigma^{\eps}$ and $\sigma^0$ are indistinguishable by EIT measurements with electrodes at $x_0, \dots, x_N$ modeled by the PEM. In other words, the relative PEM measurement map $M(\sigma^{\eps})$ defined by \eqref{Mope} satisfies
$$
M(\sigma^{\eps}) I = 0 
$$
for all electrode current patterns $I \in \R^{N+1}_\diamond$.
\end{remark}

The following lemma shows that $F^{\eps}$ is a contraction as required by the analysis preceding Proposition~\ref{propoConstruction}.
\begin{lemma}\label{lemmaTech}
Let $\gamma>0$ be a fixed parameter.   
Then, there exists $\eps_0>0$ such that for all $\eps\in(0;\eps_0]$, the map $F^{\eps}$ is a contraction on the invariant set $\mathbb{B}_\gamma$ defined by \eqref{Bg}.
\end{lemma}

\begin{proof}
Assume $\gamma>0$ is given. Let $\tau, \, \tau' \in \mathbb{B}_\gamma$ be arbitrary, $\kappa, \, \kappa' \in \mL^\infty(\Omega)$ the corresponding perturbations defined by \eqref{ExpressKappa}, and $\tilde{u}^{\eps}_n, \, \tilde{u}^{\eps}_n{}' \in \mH^1(D) / \R$ the associated solutions to (\ref{pbDefOrder1}), with $\sigma^\eps = \sigma^0 + \eps \kappa$ and $\sigma^\eps = \sigma^0 + \eps \kappa'$, respectively. By subtracting the equations defining $\tilde{u}^{\eps}_n$ and $\tilde{u}^{\eps}_n{}'$ and employing the short hand notation $w = \tilde{u}^{\eps}_n - \tilde{u}^{\eps}_n{}'$, we have
\begin{equation}
\label{help}
\div\big(\sigma^\eps \nabla w \big)  =  -\div\big((\kappa - \kappa') \nabla (u^0_n + \eps \tilde{u}^{\eps}_n{}') \big)  \quad\mbox{in }D,\qquad \nu \cdot \sigma^0\nabla w = 0\quad\mbox{on }\partial D.
\end{equation}
From \eqref{interiorR}, \eqref{pertnorm}, \eqref{perturbedConductivity} and \eqref{ExpressKappa} it easily follows that
$$
\| u^0_n + \eps \tilde{u}^{\eps}_n{}' \|_{\mH^1(\Omega)/\R} \leq C
$$
uniformly for all $\eps \in [0; \eps_0]$ if $\eps_0 > 0$ is chosen small enough. Subsequently, an application of the Lax--Milgram lemma to \eqref{help} results in the estimate
$$
\| \tilde{u}^{\eps}_n - \tilde{u}^{\eps}_n{}' \|_{\mH^1(\Omega)/\R} \leq C \| \kappa - \kappa' \|_{\mL^\infty(\Omega)} \| u^0_n + \eps \tilde{u}^{\eps}_n{}' \|_{\mH^1(\Omega)/\R} \leq C \| \kappa - \kappa' \|_{\mL^\infty(\Omega)} \leq C |\tau-\tau'|
$$   
where all occurrences of $C > 0$ are independent of $\eps \in [0; \eps_0]$.

In particular,
$$
\big| \tilde{\mathscr{M}}^\eps(\tau) - \tilde{\mathscr{M}}^\eps(\tau') \big|
\leq C |\tau-\tau'|
$$
by virtue of \eqref{defTermContra}, the Cauchy--Schwarz inequality, \eqref{interiorR} and \eqref{pertnorm}.  
The definition (\ref{DefMapContra}) then implies that
\begin{equation}\label{EstimContra}
|F^{\eps}(\tau)-F^{\eps}(\tau')|\le C\,\eps\,|\tau-\tau'| \qquad \textrm{for all} \ \tau,\tau'\in\mathbb{B}_\gamma,
\end{equation}
where $C >0$ is independent of $\eps \in [0; \eps_0]$. Moreover, noting that
$|F^{\eps}(0)| \le C\,\eps$ due to \eqref{defTermContra}, \eqref{interiorR} and \eqref{pertnorm}, we deduce from \eqref{EstimContra} that also
\begin{equation}
\label{invariant}
|F^{\eps}(\tau)| \le C\,\eps
\end{equation}
for all $\tau\in\mathbb{B}_\gamma$. Reducing $\eps_0$ if necessary, \eqref{EstimContra} and \eqref{invariant} finally show that the map $F^{\eps}$ is a contraction on the invariant set $\mathbb{B}_\gamma$ for all $\eps \in [0; \eps_0]$.
\end{proof}

Let us denote by $\tau^{\mrm{sol}}\in\mathbb{B}_\gamma$ the unique solution of the problem (\ref{PbFixedPoint}). As a side product, the previous proof ensures the existence of a constant $C_0>0$ independent of $\eps \in (0; \eps_0]$ such that 
\begin{equation}\label{RelAPos}
|\tau^{\mrm{sol}}| = |F^{\eps}(\tau^{\mrm{sol}})| \le C_0\,\eps \qquad \textrm{for all} \ \eps\in(0;\eps_0].
\end{equation}
Combined with \eqref{ExpressKappa}, this tells us that $\kappa$ is equal to $\kappa_0$ at the first order.

At this stage, the remaining job consists in showing that there are functions $\kappa_0$, $\kappa_{ij}$ supported in $\overline{\Om}$ satisfying (\ref{Conditions1})--(\ref{Conditions2}). The following lemma is a classical result on dual basis (see, e.g., \cite[Lemma 4.14]{Kres14}); we present its proof here because it will be utilized in the algorithm of Section~\ref{sec:algo}.

\begin{lemma}\label{lemmaLinearIndep}
There are functions $\kappa_0, \, \kappa_{ij} \in L^\infty(D)$ supported in $\overline{\Om}$ and satisfying (\ref{Conditions1})--(\ref{Conditions2}) if and only if $\{\nabla u^0_{i}\cdot \nabla u^0_{j}\}_{1\le i\le j\le N}$ is a family of linearly independent functions on $\Om$.
\end{lemma}
\begin{proof}
$\star$ Assume there are functions $\kappa_{ij}$ verifying (\ref{Conditions2}). If $\alpha_{ij}$ are real coefficients such that 
\[
\sum_{j=1}^{N}\sum_{i=1}^{j} \alpha_{ij}\,\nabla u^0_{i}\cdot \nabla u^0_{j} =0\qquad\mbox{in }\Om,
\]
then multiplying by $\kappa_{i'j'}$ and integrating over $D$, we find that $\alpha_{i'j'}=0$. This allows to show that $\alpha_{ij}=0$ for all $1\le i \le j \le N$ and proves that the family $\{\nabla u^0_{i}\cdot \nabla u^0_{j}\}_{1\le i\le j\le N}$ is linearly independent. 

$\star$ Now, assume that $\{\nabla u^0_{i}\cdot \nabla u^0_{j}\}_{1\le i\le j\le N}$ is a family of linearly independent functions on $\Om$. To simplify the notation, we introduce the auxiliary functions $\psi_k\in\mrm{L}^2_{\mrm{loc}}(D)$, $k=1,\dots ,K=N(N+1)/2$, such that
\[
\begin{array}{ccc}
\psi_1 = \nabla u^0_1\cdot\nabla u^0_1,\\[4pt]
\psi_2 = \nabla u^0_2\cdot\nabla u^0_1, & \quad\psi_3 = \nabla u^0_2\cdot\nabla u^0_2,\\[4pt]
\psi_4 = \nabla u^0_3\cdot\nabla u^0_1, & \quad\psi_5 = \nabla u^0_3\cdot\nabla u^0_2, & \qquad\dots\quad, 
\end{array}
\]
as well as the symmetric matrix $\mathbb{A} \in \R^{K \times K}$ given elementwise by
\[
\mathbb{A}_{kk'} = \int_{\Om} \psi_{k}\,\psi_{k'} \,d x, \qquad k,k' = 1, \dots, K.
\]
If $a = (\alpha_1,\dots,\alpha_K)^{\top}\in\ker\,\mathbb{A}$, defining $v=\sum_{k=1}^K\alpha_{k}\,\psi_{k}$, we find that 
$$
\int_{\Om} v^2 \,d x = a^{\top} \! \mathbb{A} \, a = 0.
$$ 
We deduce that $v=0$ in $\Om$, and thus $\alpha_1 = \cdots = \alpha_K =0$ because, by assumption, the family $\{\psi_k\}_{1\le k \le K}$ is linearly independent. As a consequence, $\mathbb{A}$ is invertible. 

Consider the functions $\tilde{\kappa}_1,\dots,\tilde{\kappa}_K$ defined via
\begin{equation}\label{defTermInvert}
\tilde{\kappa}_k = \sum_{k'=1}^{K}\mathbb{A}_{kk'}^{-1}\,\psi_{k'}|_{\Om}\quad\mbox{in }\Om\qquad\mbox{ and }\qquad \tilde{\kappa}_k = 0 \quad\mbox{in }D\setminus\overline{\Om}.
\end{equation}
In particular, $\tilde{\kappa}_k \in \mL^\infty(D)$ because $\psi_{k'}$ are smooth in a neighborhood of $\Omega$.
We have 
\[
\int_{D} \tilde{\kappa}_k\,\psi_{l}\,dx = \sum_{k'=1}^{K} \mathbb{A}_{kk'}^{-1}\mathbb{A}_{k'l} = \begin{array}{|ll} 1 & \mbox{ if } k=l,\\
0 & \mbox{ else}.\\
\end{array}
\]
By renumbering $\tilde{\kappa}_k$, this proves the existence of functions $\kappa_{ij}$ verifying the conditions (\ref{Conditions2}).  Finally, we construct $\kappa_0\not\equiv0$ that satisfies (\ref{Conditions1}): 
\begin{equation}\label{projKappa0}
\kappa_0 = \kappa^{\#}_0 - \sum_{j=1}^{N}\sum_{i=1}^{j} \left(\int_{D} \kappa^{\#}_0\,\kappa_{ij} \,d x\right)\,\kappa_{ij},
\end{equation}
where $\kappa^{\#}_0$ is an arbitrary $L^\infty$-function such that $\mrm{supp}(\kappa^{\#}_0)\subset\overline{\Om}$ and $\kappa^{\#}_0\notin \mrm{span}\{\kappa_{ij}\}_{1\le i\le j\le N}$.
\end{proof}

\section{Construction of invisible conductivities in two dimensions}
\label{sec:2D}

In this section, we study the two-dimensional case with 
\[
\sigma^0\equiv1. 
\]
In order to complement Proposition~\ref{propoConstruction}, our goal is to demonstrate that there exist functions $\kappa_0, \, \kappa_{ij} \in L^\infty(D)$ supported in a given Lipschitz domain $\Omega \Subset D$ and satisfying \eqref{Conditions1}--\eqref{Conditions2}. According to Lemma~\ref{lemmaLinearIndep}, this is equivalent to showing that $\{\nabla u^0_{i}\cdot \nabla u^0_{j}\}_{1\le i\le j\le N}$ is a family of linearly independent functions on $\Om$. Before setting to work, we remind the reader that the definition of $u^0_{n}$ can be found in (\ref{pbRef0}) and that $x_0,\dots, x_{N}$ are $N+1$ distinct  points located on $\partial D$ and corresponding to the positions of the electrodes.

\subsection{The case of  disk}
 
\begin{proposition}\label{propoLinearIndDisk}
Assume that $\sigma^0\equiv1$ and let $D \subset \R^2$ be the open unit disk. Then, $\{\nabla u^0_{i}\cdot \nabla u^0_{j}\}_{1\le i\le j\le N}$ is a family of linearly independent functions on any nonempty Lipschitz domain $\Om \Subset D$.
\end{proposition}

\begin{proof}
In this simple geometry, it is known that (see, e.g., \cite{HaHH11})
\begin{equation}\label{defFromNeumannFunction}
u^0_n(x) = v_n(x)-v_0(x),\qquad  x\in D, \ \ n = 1, \dots, N,
\end{equation}
where $v_n$ is defined by
\begin{equation}\label{defvk}
v_n(x)= - \frac{1}{\pi}\,\ln|x-x_n|,\qquad x\in D.
\end{equation}
Let $\alpha_{ij} \in \R$ be such that 
\begin{equation}\label{relationInitLinearInd}
\sum_{j=1}^{N}\sum_{i=1}^{j}\alpha_{ij} \,\nabla u^0_i\cdot \nabla u^0_j =0\qquad \mbox{in }\Om, 
\end{equation}
which can be written out explicitly:
\begin{equation}\label{SystemStep0}
\sum_{j=1}^{N}\sum_{i=1}^{j}\alpha_{ij}  \left(\frac{x-x_i}{|x-x_i|^2}-\frac{x-x_0}{|x-x_0|^2}\right)\cdot\left(\frac{x-x_j}{|x-x_j|^2}-\frac{x-x_0}{|x-x_0|^2}\right)=0 \qquad \mbox{in }\Om.
\end{equation}
By analyticity, \eqref{SystemStep0} holds in fact in all of $\R^2\setminus\mathscr{E}$, where $\mathscr{E}:=\cup_{n=0}^{N}\{x_n\}$. 
Multiplying \eqref{SystemStep0} by $|x-x_j|^2$ and letting $x$ tend to $x_j$, we see that $\alpha_{jj}=0$ for all  $j=1, \dots, N$. Moreover, multiplying by $|x-x_0|^2$ and letting  $x$ go to $x_0$, we obtain the relation 
\begin{equation}\label{relationSum}
\sum_{j=2}^{N}\sum_{i=1}^{j-1}\alpha_{ij}=0.
\end{equation}

Let us now introduce 
\begin{equation}\label{EquaUsingSum}
w := \sum_{j=2}^{N}\sum_{i=1}^{j-1}\alpha_{ij} \,
u^0_i\,u^0_j = \sum_{j=2}^{N}\sum_{i=1}^{j-1}\alpha_{ij} \left(v_i\,v_j-v_0\,(v_i+v_j)\right)\qquad\mbox{in  }\R^2\setminus\mathscr{E}, 
\end{equation}
where the second equality is a consequence of (\ref{defFromNeumannFunction}) and  (\ref{relationSum}). Using (\ref{relationInitLinearInd}) and the fact that $v_n$ defined in (\ref{defvk}) is harmonic in $\R^2\setminus\{x_n\}$, we find that $\Delta w=0$ in $\R^2\setminus\mathscr{E}$. Observing that $v_n$ is a multiple of the fundamental solution for the Laplacian centered at $x_n$, a standard computation then yields
\[
\Delta w = - 2\sum_{j=2}^{N}\sum_{i=1}^{j-1}\alpha_{ij}\,\Big((v_j(x_i)-v_0(x_i))\,\delta_i +(v_i(x_j)-v_0(x_j))\,\delta_j
-(v_i(x_0)+v_j(x_0))\,\delta_0\Big) \qquad {\rm in} \ \R^2.
\]
Motivated by this expression, we introduce another auxiliary function, namely 
\begin{equation}\label{decompow2}
 \tilde{w} = \sum_{j=2}^{N}\sum_{i=1}^{j-1}\alpha_{ij}\,\Big((v_j(x_i)-v_0(x_i))\,v_i +(v_i(x_j)-v_0(x_j))\,v_j
-(v_i(x_0)+v_j(x_0))\,v_0\Big),
\end{equation}
which obviously satisfies $\Delta \tilde{w} = \Delta w$ in $\R^2$. 
Since $v_n$ belongs to $\mathscr{C}^{\infty}(\R^2\setminus\{x_n\})$, it is clear that $w-\tilde{w}\in\mH^{1-\eta}_{\mrm{loc}}(\R^2)$ for all $\eta>0$. Furthermore, at infinity, it holds that $w-\tilde{w}=o(|x|)$. This allows to prove that $w-\tilde{w}$ is a tempered distribution (see, e.g., \cite[Chapter VII]{Horm03}) on $\R^2$. 

A classical extension of the Liouville theorem for harmonic functions (\cite[Chapter 3, Proposition 4.6]{Tayl11}) indicates that the only harmonic tempered distributions in $\R^d$ are the harmonic polynomials; in particular, any harmonic tempered distribution that behaves as $o(|x|)$ at infinity is a constant. This implies that $w-\tilde{w}=C$ in $\R^2$ for some $C \in \R$. On the other hand, reordering the terms in (\ref{EquaUsingSum}) and (\ref{decompow2}), we obtain the representation
\[
w-\tilde{w} = v_N\sum_{i=1}^{N-1}\alpha_{iN} \,\Big((v_i-v_i(x_N))-(v_0-v_0(x_N))\Big)+\hat{w}_N
\]
where $\hat{w}_N$ is a function that is analytic in a neighborhood of $x_N$. We conclude that 
$$
v_N\sum_{i=1}^{N-1}\alpha_{iN} \,\Big((v_i-v_i(x_N))-(v_0-v_0(x_N))\Big)
$$ 
must also be analytic about $x_N$. Because $v_N$ admits a logarithmic singularity at $x_N$, this is possible if and only if
\begin{equation}\label{expressionZero}
\sum_{i=1}^{N-1}\alpha_{iN} ((v_i-v_i(x_N))-(v_0-v_0(x_N)))=0
\end{equation}
in a neighborhood of $x_N$. By analytic continuation, (\ref{expressionZero}) must hold on all of $\R^2\setminus\mathscr{E}$. Due to the singular behavior of the function $v_i$ at $x_i$, this means that $\alpha_{iN}=0$ for all $i=1,\dots,N-1$.

Studying the behavior of $w-\tilde{w}$ successively at $x_{N-1},\dots,x_2$, one shows analogously that $\alpha_{ij}=0$ for all $1 \le i< j\le N-2$. As the diagonal elements $\alpha_{jj}$, $j=1,\dots, N$, were deduced to vanish already at the beginning of the proof, this shows that $\{\nabla u^0_{i}\cdot \nabla u^0_{j}\}_{1\le i\le j\le N}$ is a family of linearly independent functions on $\Om$.
\end{proof}

\subsection{General case in two dimensions}

With the help of conformal mappings, Proposition~\ref{propoLinearIndDisk} can be generalized to the case of arbitrary smooth two-dimensional domains.
\begin{proposition}\label{propoLinearInd2DGene}
Assume that $\sigma^0\equiv1$. Let $D \subset \R^2$ be a simply connected and bounded domain with a $\mathscr{C}^{\infty}$-boundary. Then, $\{\nabla u^0_{i}\cdot \nabla u^0_{j}\}_{1\le i\le j\le N}$ is a family of linearly independent functions on any nonempty Lipschitz domain $\Om \Subset D$.
\end{proposition}

\begin{proof}
Denote the open unit disk by $B \subset \R^2$ and let $\Phi: D \to B$ be a conformal map of $D$ onto $B$. As $\partial D$ is smooth, $\Phi$ also defines a smooth diffeomorphism of $\partial D$ onto $\partial B$ \cite{Pommerenke92}. Let $\hat{u}^0_n \in \mH^{1-\eta}(B)/\R$, $\eta>0$, be the unique solution of
\begin{equation*}
\Delta \hat{u}^0_n  =  0 \quad\mbox{in }B,\qquad \nu \cdot \nabla \hat{u}^0_n = \hat{\delta}_n - \hat{\delta}_0 \quad\mbox{on }\partial B,
\end{equation*}
where $\hat{\delta_i} \in \mH^{-1/2-\eta}(\partial B)$, $i=0, \dots, N$, denotes the Dirac delta distribution supported at $\Phi(x_n) \in \partial B$. According to \cite[Proof~of~Theorem~3.2]{HaHH11}, it holds that
\begin{equation}\label{eq:conf}
u^0_{n} = \hat{u}^0_n \circ \Phi \qquad {\rm in} \ D
\end{equation}
for all $n=1, \dots, N$.

Let $\alpha_{ij} \in \R$ be such that 
\begin{equation*}
\sum_{j=1}^{N}\sum_{i=1}^{j}\alpha_{ij} \,\nabla u^0_i\cdot \nabla u^0_j =0\qquad \mbox{in }\Om.
\end{equation*}
Due to the harmonicity of $u^0_n$, $n=1,\dots,N$, in $D$, we obtain
\begin{equation*}
\Delta \Big( \sum_{j=1}^{N}\sum_{i=1}^{j}\alpha_{ij} \, \big(\hat{u}^0_i \, \hat{u}^0_j \big) \circ \Phi \Big) = \Delta \Big( \sum_{j=1}^{N}\sum_{i=1}^{j}\alpha_{ij} \,u^0_i \, u^0_j\Big) = 0
\qquad  \mbox{in }\Om.
\end{equation*}
Since a composition with the conformal map $\Phi^{-1}$ retains harmonicity, it holds that
$$
2 \sum_{j=1}^{N}\sum_{i=1}^{j}\alpha_{ij} \,\nabla \hat{u}^0_i\cdot \nabla \hat{u}^0_j = \Delta \Big( \sum_{j=1}^{N}\sum_{i=1}^{j}\alpha_{ij} \, \hat{u}^0_i \, \hat{u}^0_j \Big) = 0 \qquad {\rm in} \ \Phi(\Om).
$$
Because the family $\{\nabla \hat{u}^0_{i}\cdot \nabla \hat{u}^0_{j}\}_{1\le i\le j\le N}$ is linearly independent on $\Phi(\Om)$ by Proposition~\ref{propoLinearIndDisk}, it follows that
$\alpha_{ij} = 0$ for all $1 \leq i \leq j \leq N$ and the proof is complete.
\end{proof}

Combining Proposition \ref{propoConstruction}, Lemma \ref{lemmaLinearIndep} and Proposition \ref{propoLinearInd2DGene}, we obtain the main result of this article.
\begin{theorem}\label{MainTheorem}
Assume that $\sigma^0\equiv1$. Let $D \subset \R^2$ be a simply connected and bounded domain with a $\mathscr{C}^{\infty}$-boundary. For any compactly embedded Lipschitz domain $\Om \Subset D$, there exists a conductivity $\sigma^{\eps}\in\mrm{L}^{\infty}(D)$, with $\sigma^\eps \geq c >0$, such that $\sigma^{\eps}-\sigma^{0}\not\equiv0$, $\mbox{supp}(\sigma^{\eps}-\sigma^{0})\subset \overline{\Om}$ and
\[
\mathscr{M}_{i,j}(\sigma^\eps) = \langle \delta_i-\delta_{0},(\Lambda^\eps-\Lambda^0)(\delta_j-\delta_{0})\rangle_{\partial D}=0
\]
for all $i,j=1,\dots,N$. In other words, 
$$
M(\sigma^{\eps})I = 0 \qquad \mbox{for all} \ I \in \R^{N+1}_\diamond,
$$
where $M(\sigma^\eps): \R^{N+1}_\diamond \to \R^{N+1}/ \R$ is the relative PEM measurement map from \eqref{Mope} corresponding to the electrode locations $x_0, \dots, x_N \in \partial D$.
\end{theorem}

\begin{remark}
In \cite{Hyvonen13,Seis14} it has been shown that $\mathscr{M}(\sigma^\eps)=0$ if and only if $\sigma^\eps=\sigma^0 (\equiv 1)$ when the number of electrodes is countably infinite in two dimensions. Here, we have demonstrated that one can constructively find $\sigma^\eps\not\equiv\sigma^0$ such that $\mathscr{M}(\sigma^{\eps})=0$ in case the number of electrodes is finite.
\end{remark}

\section{Algorithmic implementation}
\label{sec:algo}
For a given Lipschitz domain $\Om $ such that $\overline{\Om}\subset D$, functions $\kappa_0$, $\kappa_{ij}$ satisfying \eqref{Conditions1}--\eqref{Conditions2} can be constructed by following the line of reasoning in the second part of the proof of Lemma~\ref{lemmaLinearIndep}, assuming the potentials $u_n^0$ verifying \eqref{pbRef0} are available. Notice that the bounded function $\kappa^{\#}_0$ with ${\rm supp(\kappa^{\#}_0)}\subset \overbar{\Omega}$ appearing in \eqref{projKappa0} can be chosen arbitrarily as long as it does not belong to ${\rm span}\{\kappa_{ij}\}_{1\leq i \leq j \leq N}$. As a consequence, finding $ \kappa_0^\#$ is almost trivial and allows a lot of freedom. 

We denote by $\tau^{k}$, $\kappa^k$, $u^{\eps, k}_n$, $\tilde{u}^{\eps, k}_n$ the realizations of $\tau$, $\kappa$,  $u^{\eps}_n$, $\tilde{u}^{\eps}_n$ at iteration $k\ge0$ of the algorithm; the aforementioned entities are introduced in \eqref{defTermContra}, \eqref{ExpressKappa}, \eqref{pbRef}, \eqref{decompoUeps}, respectively. Moreover, we set $\sigma^{\eps, k}=\sigma^0+\eps\kappa^k$. Using formulas \eqref{defTermContra} and \eqref{DefMapContra}, we recursively define 
\[
\tau^{k+1}_{ij} = -\eps\int_{D}\kappa^k\,\nabla \tilde{u}^{\eps,k}_i\cdot \nabla u^0_j \,d x, \qquad k \geq 0, \ 1\le i \le j \le N.
\]
Since $\eps\tilde{u}^{\eps,k}_n = u^{\eps,k}_n-u^{0}_n$ by \eqref{decompoUeps}, we obtain
\begin{equation}
\label{FPiter}
\tau^{k+1}_{ij} = \tau^{k}_{ij} - \int_{D}\kappa^k\,\nabla u^{\eps,k}_i\cdot \nabla u^0_j \,d x, \qquad k \geq 0, \ 1\le i \le j \le N,
\end{equation}
due to \eqref{Conditions1} and \eqref{Conditions2}.
In particular, remark that $|\tau^{k+1}_{ij}- \tau^{k}_{ij}| \le \eta$ implies $\mathscr{M}_{i j}(\sigma^{\eps, k})\le \eps\,\eta$ for any $\eta > 0$ by virtue of \eqref{eqnMeasure}.
To sum up, our algorithm for computing invisible conductivity perturbations is as follows:

\begin{alg}\label{alg:1}
Assume that the potentials $u_n^0$, $n=1,\dots, N$, are available.
Construct $\{\kappa_{ij}\}_{1\leq i \leq j \leq N}$, choose $\kappa_0^\# \notin {\rm span}\{\kappa_{ij}\}_{1\leq i \leq j \leq N} $ and compute $\kappa_0$. Select $\eps > 0$ and $\tau^0 \in S^N$. Run the fixed point iteration \eqref{FPiter} until the desired stopping criterion is met. Construct the invisible conductivity perturbation via \eqref{ExpressKappa}. If divergent behavior is observed, decrease $\eps$. 
\end{alg}

\begin{figure}[t!]
\begin{center}
\includegraphics[width=1\textwidth]{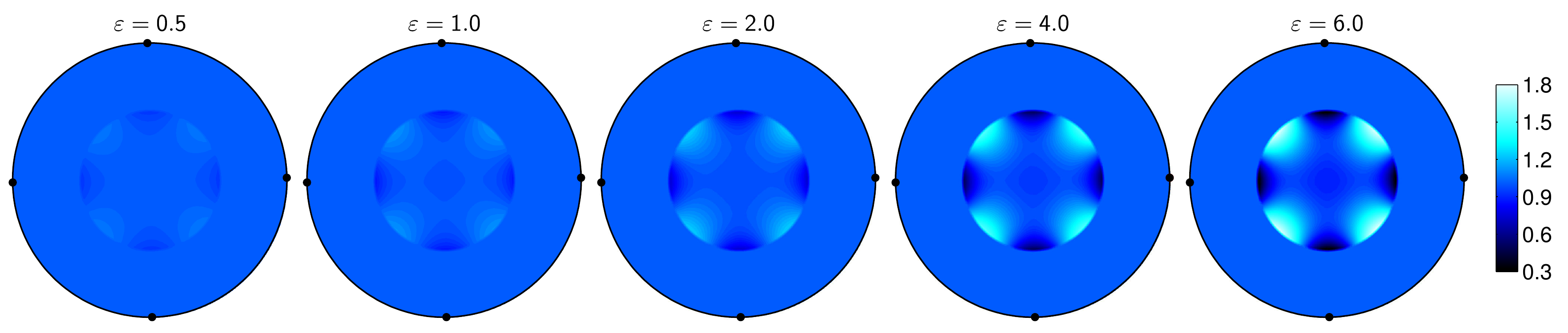} \\[0.2cm]
\raisebox{0.0cm}{\includegraphics[width=0.6\textwidth]{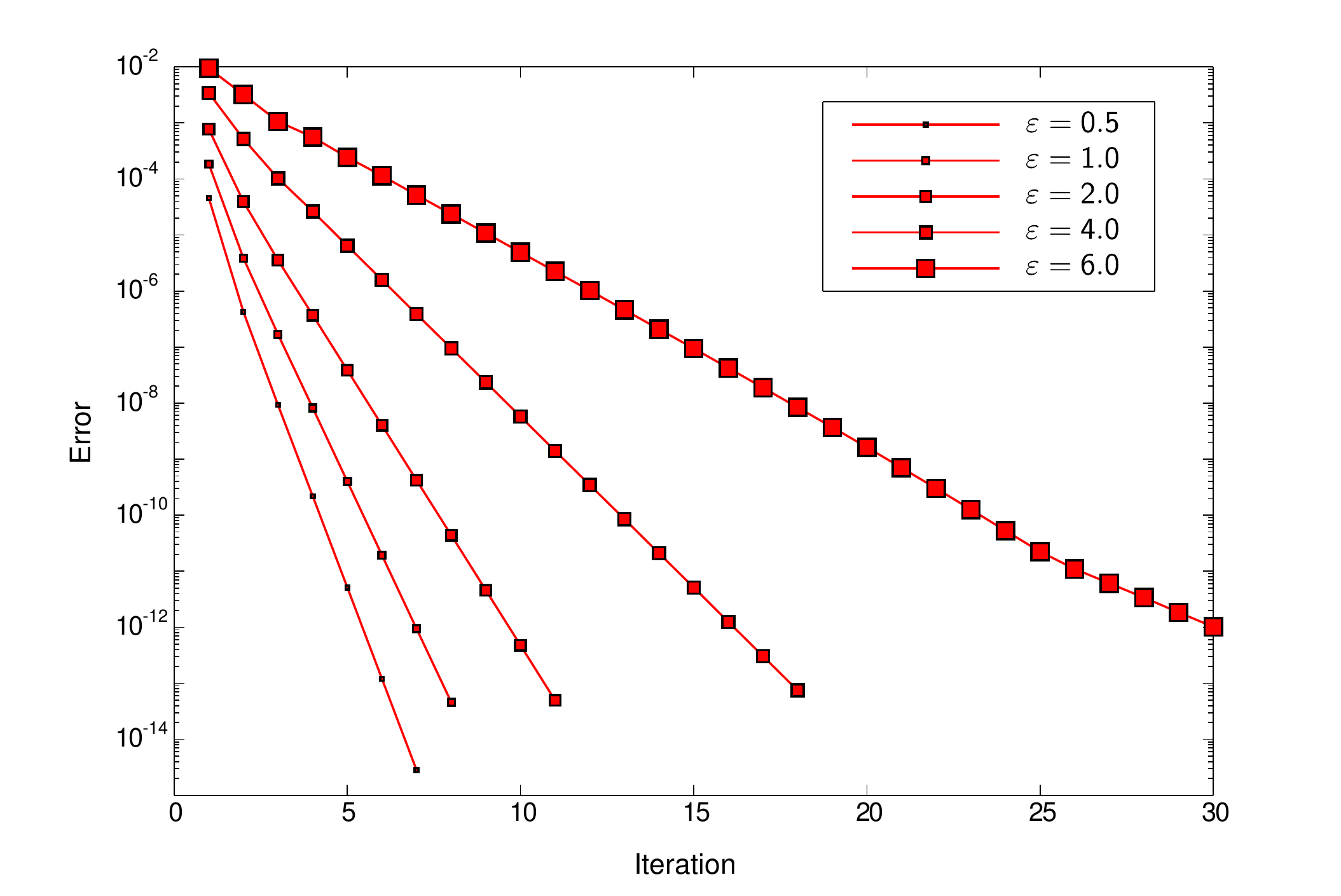}}
\raisebox{0.2cm}{\includegraphics[width=0.38\textwidth]{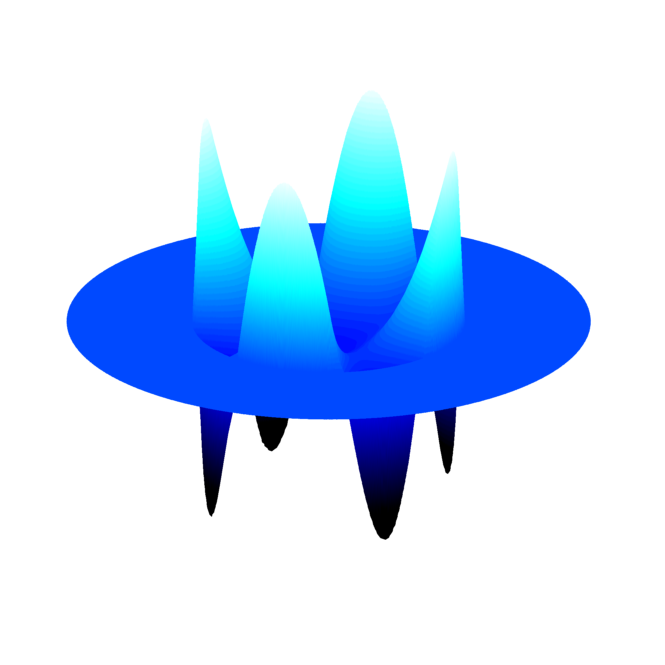}}
\caption{\label{fig:1}Top row: Outputs $\sigma^\eps$ of Algorithm~\ref{alg:1} with different values of the perturbation size parameter $ \eps > 0$. The point-like electrodes are located at the polar angles $ \theta = 1^ \circ, 91^ \circ, 181^ \circ, 271^ \circ $. Bottom left: The discrepancy \eqref{eq:tauerror} between consecutive iterates of Algorithm \ref{alg:1}. Bottom right: A surface plot of $\sigma^\eps$ corresponding to $ \eps = 6.0 $. } 
\end{center}
\vspace{-2mm}
\end{figure}

\section{Numerical experiments}
\label{sec:numerics}

In this section, we implement Algorithm~\ref{alg:1} by resorting to {\em finite element} (FE) approximations of the PEM. We are interested in validating the convergence of the fixed point iteration \eqref{FPiter} and visualizing the output conductivity $\sigma^\eps$ (cf.~\eqref{perturbedConductivity}) for different subdomains $\Omega$, `initial guesses' $\kappa_0^\#$ and (point) electrode configurations. It turns out that especially the choice of $\kappa_0^\#$ has a considerable effect on the output, which means that Algorithm~\ref{alg:1} can straightforwardly be used to construct several invisible conductivity perturbations for a given measurement configuration and $\Omega$.  The degree of indistinguishability of the produced $\sigma^\eps$ compared to the background conductivity $\sigma^0$  is also tested in the framework of the more realistic CEM: According to our simulations, the relative CEM measurements corresponding to smallish electrodes and the constructed conductivities fall below any reasonable measurement noise level. 

We only consider the homogeneous reference conductivity $\sigma^0\equiv1$ and choose $D\subset\R^2$ to be the unit disk, meaning that the functions $u^0_n$ satisfying \eqref{pbRef} are explicitly given by \eqref{defFromNeumannFunction}. Due to the Riemann mapping theorem, this geometric simplification does not severely reduce the generality of our (two-dimensional) numerical experiments: In any smooth simply connected domain $D \subset \R^2$ the potentials $u^0_n$ employed in the construction of $\kappa_{ij}$ can be computed using \eqref{eq:conf} as long as a conformal mapping sending $D$ onto the unit disk is available.

\begin{figure}[b!]
\begin{center}
\includegraphics[width=1\textwidth]{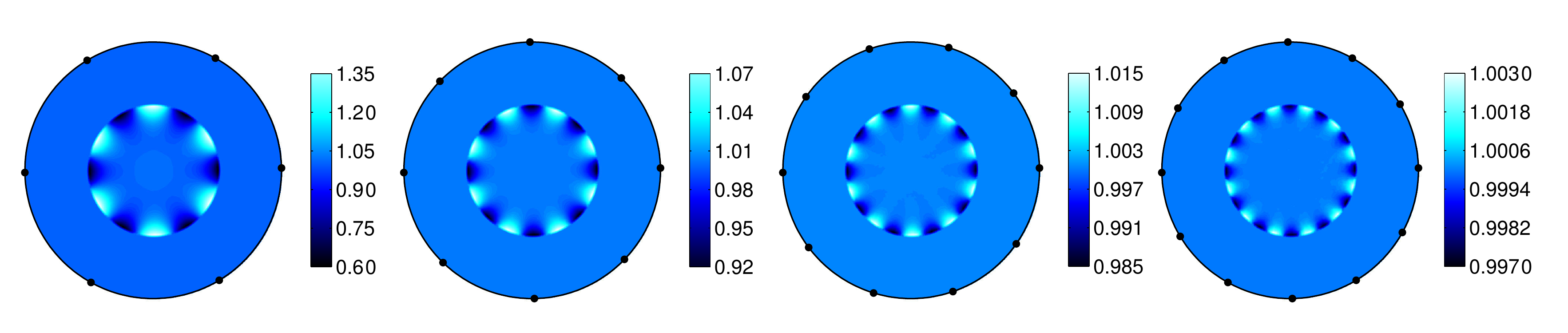}
\caption{\label{fig:2} Outputs of Algorithm~\ref{alg:1} for $ N = 5,7,9,11 $ (recall that $ N+1 $ is the number of electrodes). The point-like electrodes are located at the polar angles $ \theta_j = 1^\circ+\tfrac{j}{N+1}360^\circ $, $ j = 0,1,\ldots,N $. The values of $ \eps $ used in the computations are $14$, $12$, $8$ and $8$, respectively. 
} 
\end{center}
\vspace{-2mm}
\end{figure}

The functions $u_n^{\eps, k}$ needed when running Algorithm~\ref{alg:1} are computed at each iteration as follows: First, we (numerically) solve \eqref{pbDefOrder1} with $\sigma^{\eps} = \sigma^{\eps, k} = 1+\eps\kappa^k$ and $\kappa=\kappa^k$ to obtain $\tilde{u}^{\eps,k}_n$. We remind the reader that $\tilde{u}^{\eps,k}_n$ belongs to $\mH^1(D)/\R$ so it can be approximated using standard FE methods. Then, we set $u^{\eps,k}_n=u^{0}_n+\eps\tilde{u}^{\eps,k}_n$.
When computing $\tilde{u}^{\eps,k}_n$, we employ piecewise quadratic polynomial FE basis (Lagrange $ \mathbb{P}_2 $); the number of triangular elements is around $20\,000$ in all tests. The conductivity perturbation `shape functions' $\kappa_{ij}$ are evaluated (and interpolated) on the FE mesh with the help of \eqref{defTermInvert} and \eqref{defFromNeumannFunction}. Without exception, we use the stopping rule
\begin{equation}\label{eq:tauerror}
\sum_{i=1}^N\sum_{j=1}^N |\tau_{ij}^{k+1}-\tau_{ij}^k| < 10^{-8}
\end{equation}
for Algorithm~\ref{alg:1}, that is, the difference between consecutive iterates in the scheme \eqref{FPiter} is monitored. The choice of the threshold value $10^{-8}$ in \eqref{eq:tauerror} is a slight overkill since for practical EIT the noise level in {\em relative} measurements is typically over $1\%$ \cite{Cheney99}, and \eqref{eq:tauerror} corresponds to a considerably smaller error for all (relative) PEM data we have simulated. As the starting value for the iteration \eqref{FPiter}, we choose the null matrix $\tau^0 = 0$ in all experiments. The numerical results are presented in {\sc Figures}~\ref{fig:1}--\ref{fig:4}.

\begin{figure}[t!]
\begin{center}
\includegraphics[width=0.85\textwidth]{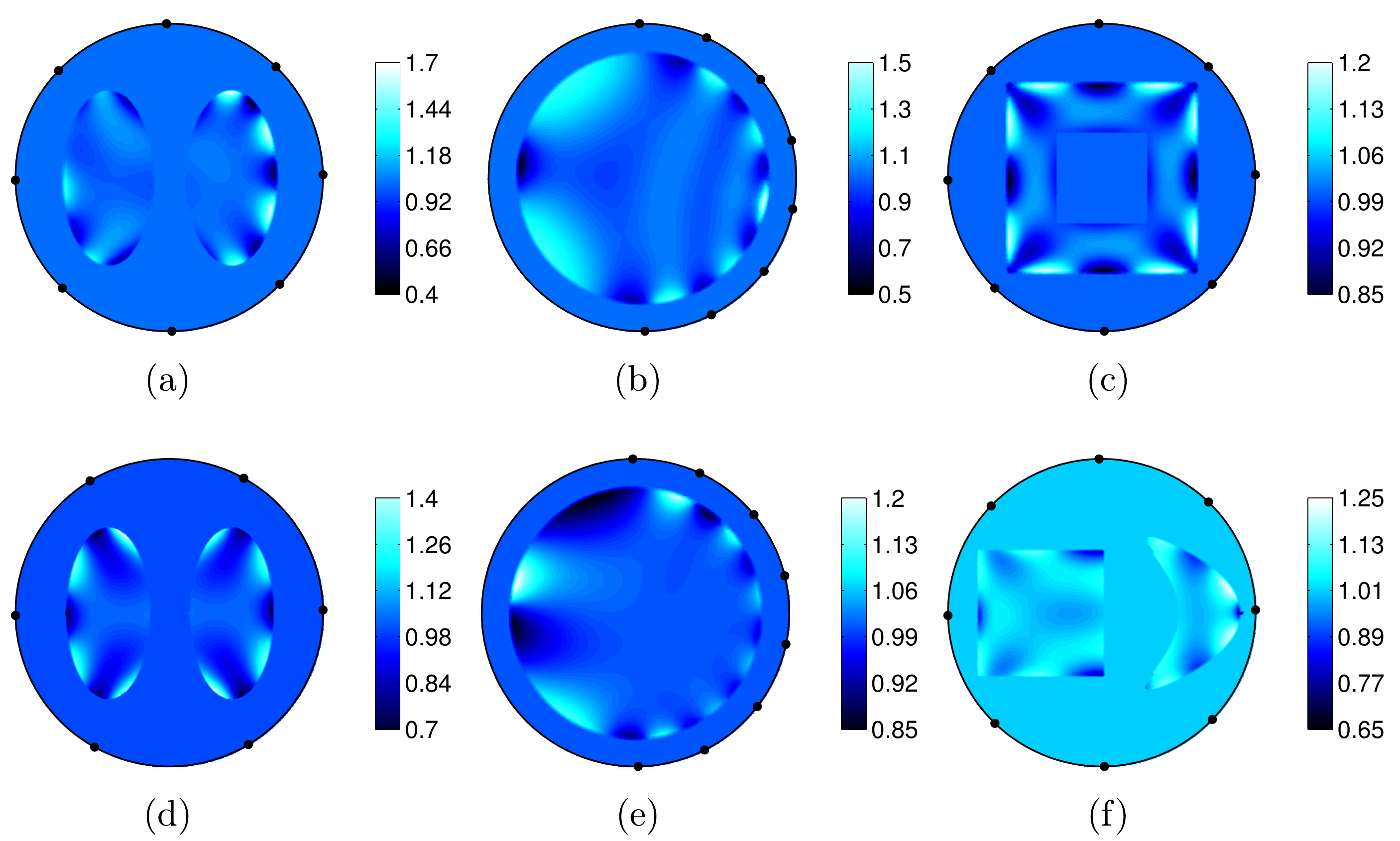}\\[8pt]
\includegraphics[width=0.4\textwidth]{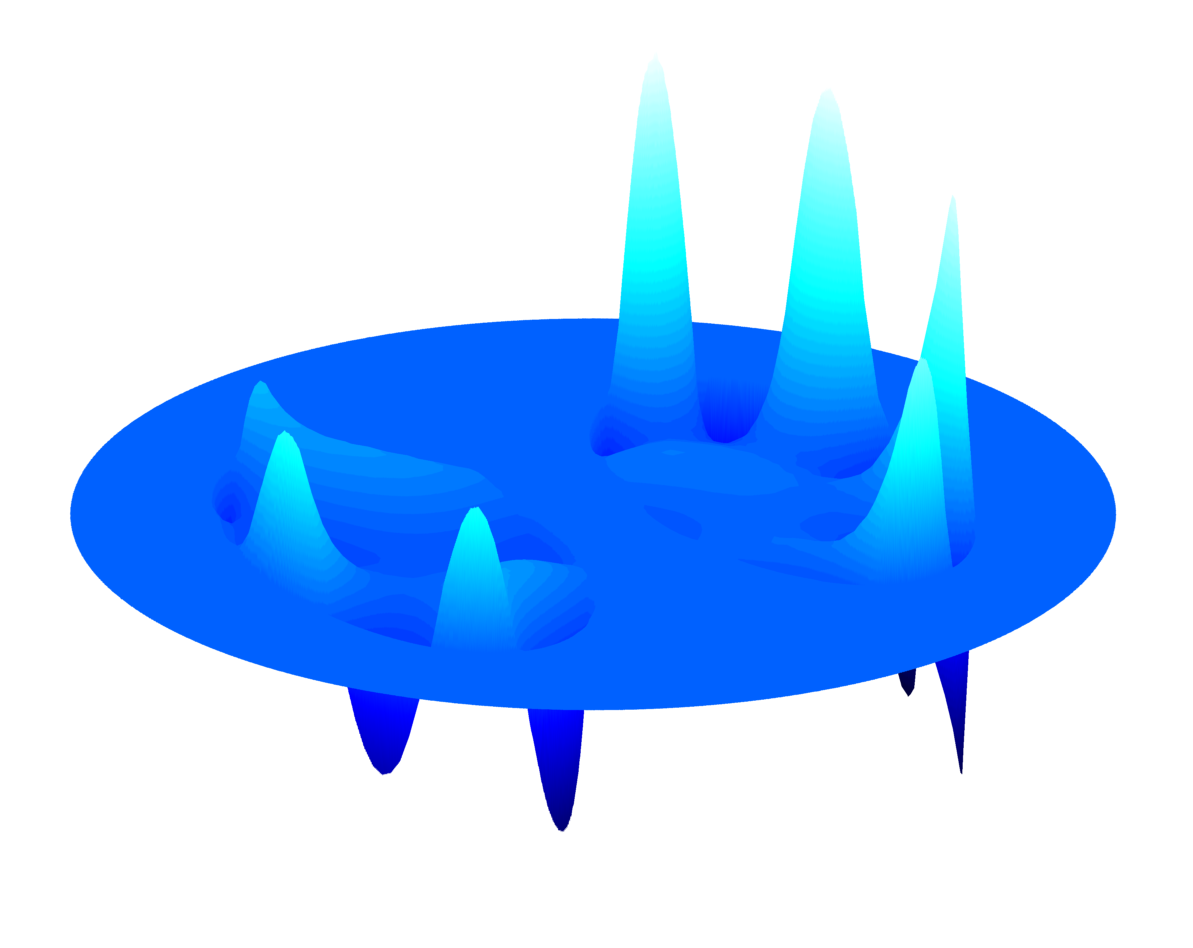}
\raisebox{2.5cm}{
\begin{tabular}{|c|c|c|c|}
\hline
\rowcolor{Gray}
 & $ \eps $ & $ \kappa_0^\#(x,y) $ & $ E_{\rm CEM}^\eps(\mathcal{I}) $ \\
\hline 
\cellcolor{Gray}(a) & $4.0$ & $ x+y+1 $ & $ 1.4\times 10^{-4} $ \\
\hline
\cellcolor{Gray}(b) & $ 2.0 $ & $\exp(-(x+0.5)^2 -y^2)$ & $ 1.1\times 10^{-3} $ \\ 
\hline
\cellcolor{Gray}(c) & $ 0.25 $ & $1$ & $ 2.3\times 10^{-4} $ \\ 
\hline
\cellcolor{Gray}(d) & $ 6.0 $ & $ 1 $ & $ 1.1 \times 10^{-4} $ \\ 
\hline
\cellcolor{Gray}(e) & $ 0.5 $ & $ -y $ & $ 8.4 \times 10^{-4} $ \\ 
\hline
\cellcolor{Gray}(f) & $ 2.0 $ & $ x $ & $ 8.8\times 10^{-4} $ \\
\hline
\end{tabular}
}
\vspace{-0.7cm}

\hspace{6.8cm}{\sc Table 1.} 

\vspace{0.6cm}

\caption{\label{fig:3} (a)--(f): Perturbed conductivities produced by Algorithm~\ref{alg:1} for the shown electrode configurations. Bottom left: A surface plot of the conductivity (a). {\sc Table}~1: The parameters $\eps$ and $\kappa_0^\#$ for (a)--(c) together with the respective relative CEM discrepancies $E^\eps_{\rm CEM}(\mathcal{I})$ defined by \eqref{eq:errormeas}. The CEM potentials are  simulated with the `trigonometric current basis', the universal electrode width $\pi/32$ and contact resistances of magnitude $0.01$~\cite{Cheney99,Somersalo92}.} 
\end{center}
\vspace{-2mm}
\end{figure}

{\sc Figure} \ref{fig:1} illustrates a first convergence test for Algorithm~\ref{alg:1}, investigating the effect of the choice for the free parameter $\eps$. There are four electrodes and a concentric disk of radius $1/2$ serves as the inclusion $\Omega$. The parameter $\kappa_0^\#$ is set to $\kappa_0^\# \equiv 1$. The convergence rate of the algorithm decreases as a function of $\eps$, but the fixed point scheme remains convergent for all $\eps$ in the interval $(0,6]$ (and beyond).

{\sc Figure}~\ref{fig:2} illustrates that increasing the number of electrodes has a significant impact on the output conductivity~$\sigma^\eps$ of Algorithm \ref{alg:1}. In particular, the spatial frequency of the angular oscillations in $\sigma^\eps$ seems to be directly linked to the number of electrodes. Moreover, at least with the chosen simple inclusion shape, the deviations of $\sigma^\eps$ from the unit background become less significant as the number of electrodes increases even if $\eps$ stays the same.

\begin{figure}[t!]
\begin{center}
\includegraphics[width=0.5\textwidth]{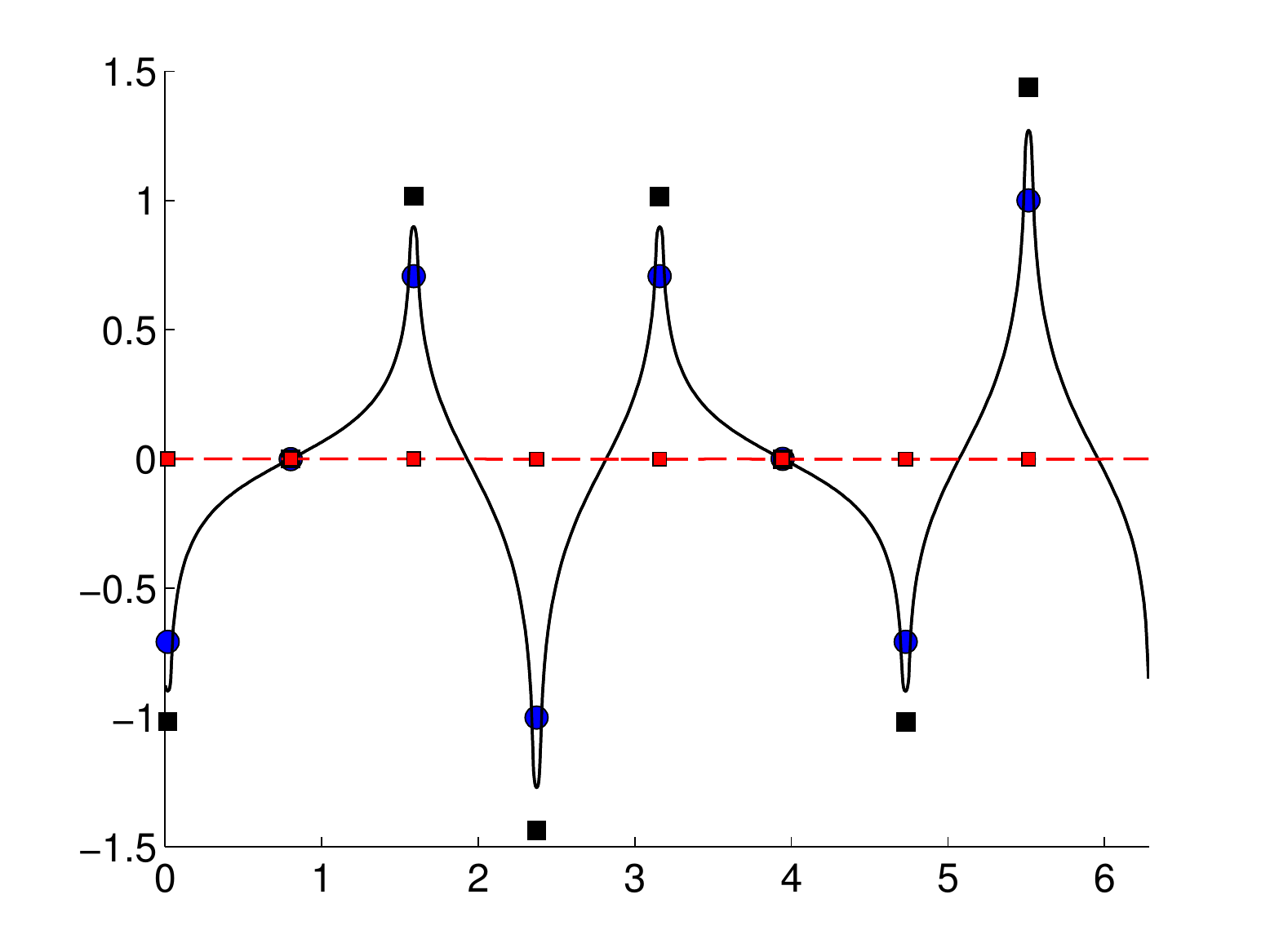}
\raisebox{8pt}{\includegraphics[width=0.48\textwidth]{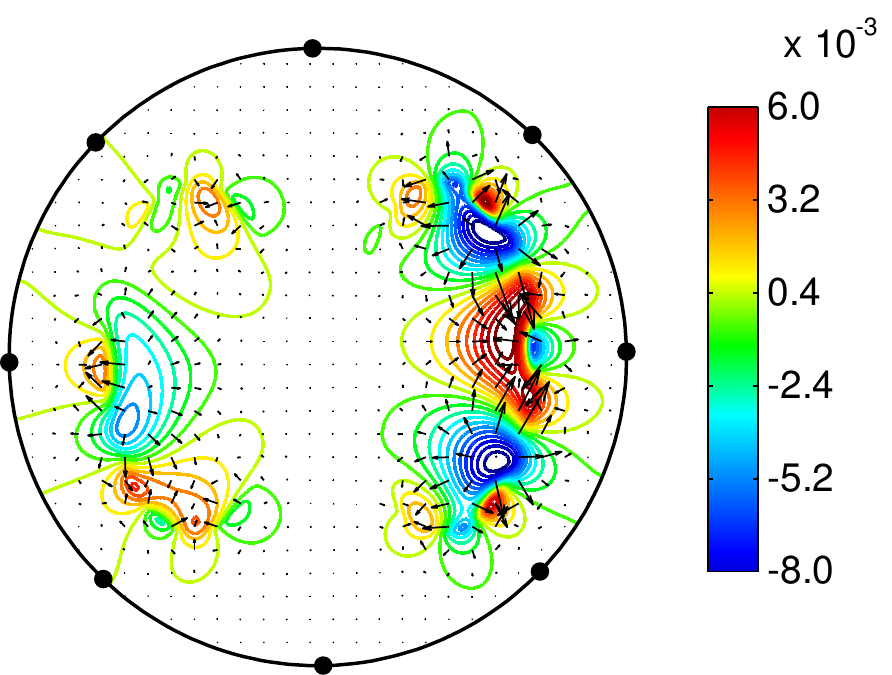}}
\caption{\label{fig:4} 
An example of CEM potentials corresponding to $\sigma^\eps$ of {\sc Figure}~\ref{fig:3}(a); see~\cite{Somersalo92}.
Left: The Dirichlet boundary value of the absolute CEM interior potential (black line), the absolute electrode voltages (black squares), the Dirichlet boundary value of the relative CEM interior potential (red dashed line) and the relative electrode potentials (red squares). The electrode-wise net input currents are also displayed (blue circles). The electrode width and contact resistances are as in {\sc Figure}~\ref{fig:3}. Right: The relative CEM interior potential. Equipotential lines and electric field arrows are displayed.} 
\end{center}
\vspace{-2mm}
\end{figure}

The inclusion shape and the choice of the initial perturbation $\kappa_0^\#$ have considerable effects on the output of the algorithm as shown in {\sc Figure}~\ref{fig:3}. A large $\Omega$ lying close to $\partial D$ typically yields significant deviations from the unit background in $\sigma^\eps$, even with a relatively high number of electrodes. Moreover, asymmetries in $\kappa_0^\#$ lead to an asymmetric $\sigma^\eps$ even in a geometrically symmetric setting. The indistinguishability of the output conductivities in {\sc Figure} \ref{fig:3} from the unit background is evaluated using FE approximations of the CEM \cite{Somersalo92,Vauhkonen97} with about $10\,000$ quadratic triangular elements, a fixed electrode width of $\pi/32$ and an underlying contact resistance parameter equal to $0.01$ on all electrodes. As a measure of the relative discrepancy for simulated (noiseless) CEM voltages, we employ 
\begin{equation}\label{eq:errormeas}
E_{\rm CEM}^\eps(\mathcal{I}) = \frac{\big|\mathcal{U}^\eps(\mathcal{I})-\mathcal{U}^0(\mathcal{I})\big|} {\big| \mathcal{U}^0(\mathcal{I}) \big|}. 
\end{equation}
Here, $\mathcal{U}^\eps(\mathcal{I}),\mathcal{U}^0(\mathcal{I}) \in \mathbb{R}^{N(N+1)}$ consist of the (stacked) CEM electrode potentials corresponding to  the `trigonometric current basis' $\mathcal{I} = \{I^{(j)}\}_{1\leq j \leq N} \subset \mathbb{R}_\diamond^{N+1}$~\cite{Cheney99} and $\sigma^\eps$, $\sigma^0$, respectively; cf.,~e.g.,~\cite{Hyvonen14}. In practice, the {\em relative} noise level in EIT measurements is {\em significantly} above $0.1\%$ (cf.~\cite{Cheney99}). Hence, {\sc Table}~1 indicates that all conductivities shown in {\sc Figure}~\ref{fig:3} are practically indistinguishable from the unit background in the framework of the CEM with smallish electrodes at the depicted locations~(cf.~\cite{Hanke11}).

Finally, {\sc Figure}~\ref{fig:4} illustrates a single simulated relative CEM potential for the conductivity $\sigma^\eps$ in {\sc Figure}~\ref{fig:3}(a) --- to be precise, both the interior electric potential and the electrode potentials are considered (cf.~\cite{Somersalo92}). The electrode widths and contact resistances are the same as in {\sc Figure}~\ref{fig:4}. It is noteworthy that both the Dirichlet boundary value of the relative interior electric potential and the relative electrode potentials do vanish according to a visual inspection, but the same does not hold for the relative electric potential in the interior of $D$. 
\begin{remark}
It is possible that the technique developed in this article could also be adapted to construct invisible conductivities directly for the CEM. However, instead of the functions $u^0_n$ defined by \eqref{pbRef0}, one would need to work with potentials corresponding to unit net currents between electrodes of finite size modeled by the CEM (cf.~\cite{Somersalo92}). This would make proving the convergence of the fixed point iteration as well as implementing the numerical algorithm more technical.
\end{remark}

\section{Concluding remarks}

We have introduced a fixed point scheme for computing invisible conductivity perturbations (of the two-dimensional unit conductivity) for a given electrode configuration in the framework of the PEM for EIT. Our numerical experiments demonstrate that the constructed conductivities are practically indistinguishable from the background also if the measurements are modeled by the more realistic CEM. In particular, take note that the conductivities shown in {\sc Figures}~\ref{fig:1}--\ref{fig:4} are as good solutions as the unit conductivity to the reconstruction problem of EIT with the considered sets of electrodes if no prior information on the behavior of the conductivity is available.

\bibliography{Bibli}
\bibliographystyle{plain}
\end{document}